\documentclass[11pt]{article}

\usepackage[english]{babel}
\usepackage[babel]{microtype}

\usepackage[margin=2.3cm]{geometry}
\usepackage{setspace}
\usepackage{fancyhdr}

\usepackage{amsmath, amssymb, amsfonts, amsthm}
\usepackage{mathtools}
\usepackage{bm}
\usepackage{bbm}
\usepackage{mathrsfs}

\usepackage{graphicx}
\usepackage{xcolor}
\usepackage{booktabs}
\usepackage{float}
\usepackage{caption}
\usepackage{tikz}
\usepackage{lineno}
\usepackage[shortlabels]{enumitem}

\usepackage{url}
\usepackage{comment}
\usepackage{todonotes}

\usepackage{hyperref}
\usepackage[capitalise]{cleveref}

\counterwithin{figure}{section}

\newtheorem{theorem}{Theorem}[section]
\newtheorem*{thm-non}{Theorem}

\newtheorem{lemma}[theorem]{Lemma}
\newtheorem{cor}[theorem]{Corollary}

\newtheorem{fact}[theorem]{Fact}

\theoremstyle{definition}

\newtheorem*{defn-non}{Definition}

\newtheorem{rmk}[theorem]{Remark}

\newcommand{\F}{\mathcal{F}}
\newcommand{\G}{\mathcal{G}}
\newcommand{\A}{\mathcal{A}}

\newcommand{\cH}{\mathcal{H}}
\newcommand{\E}{\mathbb{E}}
\newcommand{\Prb}{\mathbb{P}}
\newcommand{\Inf}{\mathrm{Inf}}

\newcommand{\Var}{\mathrm{Var}}

\newcommand{\Cov}{\mathrm{Cov}}
\newcommand{\cube}{\{0,1\}}
\newcommand\tup[1]{\left\langle #1 \right\rangle}

\DeclareMathOperator{\supp}{supp}

\title{The sharp diagonal spectral correlation inequality on the discrete cube}

\author{Fan Chang\thanks{School of Statistics and Data Science, Nankai University, Tianjin, China; and Extremal Combinatorics and Probability Group, Institute for Basic Science, Daejeon, South Korea. Email: \texttt{1120230060@mail.nankai.edu.cn}. Supported by the National Natural Science Foundation of China (NSFC) under grant 124B2019 and by the Institute for Basic Science IBS-R029-C4.}, \quad 
Hong Liu\thanks{Extremal Combinatorics and Probability Group (ECOPRO), Institute for Basic Science (IBS), Daejeon, South Korea. Email: \texttt{hongliu@ibs.re.kr}. Supported by Institute for Basic Science IBS-R029-C4.},\quad
Miao Liu\thanks{Research Center for Mathematics and Interdisciplinary Sciences, Shandong University, Qingdao, China, and Extremal Combinatorics and Probability Group (ECOPRO), Institute for Basic Science (IBS), Daejeon, South Korea. Email: \texttt{liumiao10300403@163.com}. Supported by China Scholarship Council and Institute for Basic Science IBS-R029-C4.}
}
\date{}
\begin{document}
\maketitle

\begin{abstract}
We prove the sharp diagonal spectral correlation conjecture of Friedgut,
Kahn, Kalai and Keller, proposed in their Fourier-analytic approach to Chv\'atal's conjecture. For every pair of increasing Boolean functions $f,g:\{0,1\}^n\to\{0,1\}$,
$$\mathrm{Cov}(f,g)\ge4\sum_{\varnothing\ne S\subseteq[n]}|S|\hat{f}(S)^2\hat{g}(S)^2.$$
Thus covariance controls the degree-weighted collision of the two
nonconstant Fourier spectra, giving a sharp Fourier strengthening of the
Harris--Kleitman inequality. The theorem also implies the unweighted diagonal conjecture of Friedgut--Kahn--Kalai--Keller for an increasing family and a maximal intersecting family.

The factor $4$ is optimal, and we determine all equality cases. Apart
from pairs whose relevant coordinate sets are disjoint, equality occurs only for a common dictatorship and, up to relabelling coordinates and interchanging $f$ and $g$, for the two-coordinate AND-OR pair $(f,g)=(x_i x_j,\,x_i\vee x_j).$

The main novelty is a correlated four-restriction induction and a sharp
endpoint convolution inequality. The usual two-restriction induction behind Harris--Kleitman sees only the parallel restricted pairs and loses the mixed Fourier information needed to control the degree-weighted diagonal spectral energy. We instead couple the four codimension-one restricted pairs with correlation $1/2$; this precise correlation extracts the missing degree-weighted energy as a nonnegative square.
\end{abstract}

\section{Introduction}\label{sec:intro}

How much positive correlation is forced by monotonicity? The
Harris--Kleitman inequality~\cite{Harris1960,Kleitman1966} gives the basic
qualitative answer: if two Boolean functions on a product space are increasing, then they are nonnegatively correlated. On the discrete cube this says that
$$\Cov(f,g)\ge0$$
for increasing Boolean functions $f,g:\cube^n\to\cube$, where throughout
$\{0,1\}^n$ is equipped with the uniform measure and $\Cov(f,g):=\E[fg]-\E[f]\E[g].$ Here increasing means increasing with respect to the coordinatewise order on $\cube^n$. The central quantitative problem is to strengthen this nonnegativity statement by measuring how the two functions overlap: through their influences, their Fourier spectra, or other analytic profiles.

One influential answer is Talagrand's correlation inequality~\cite{Talagrand1996correlated}. For a Boolean function $f:\cube^n\to\cube$, the influence of the $i$-th coordinate is $\Inf_i[f]:=\Prb\left[f(x)\neq f(x\oplus e_i)\right]$, where $e_i$ is the $i$-th standard basis vector and $\oplus$ denotes addition modulo $2$.
Talagrand proved that, for all increasing Boolean functions $f,g$,
\[
\Cov(f,g)\ge c\frac{\sum_{i=1}^n\Inf_i[f]\Inf_i[g]}{\log\left(e/\sum_{i=1}^n\Inf_i[f]\Inf_i[g]\right)}
\]
with a universal constant $c>0$. This initiated several variants and extensions of quantitative correlation inequalities for monotone functions; see, for example,~\cite{DNS2022,Eldan2022,KKM2016correlation,Keller09correlation,Keller2011,Keller2012,KK2019DA,KMS2014correlation}. These inequalities are closely connected with sharp thresholds, noise sensitivity, quantitative FKG-type inequalities, and Gaussian correlation inequalities.

The present paper is motivated by a different, more spectral, strengthening of Harris--Kleitman proposed by Friedgut, Kahn, Kalai, and Keller~\cite{FKKK2018correlation}. Their starting point was Chv\'atal's conjecture, a central conjecture in extremal set theory, on intersecting subfamilies of hereditary families. Recall that a family $\A\subseteq2^{[n]}$ is intersecting if $A\cap A'\neq\varnothing$ for all
$A,A'\in\A$, and hereditary if it is closed under taking subsets. Chv\'atal's conjecture~\cite{Chvatal1974} asserts that, for every hereditary family $\cH$, some largest intersecting subfamily of $\cH$ is a star: it consists of all members of $\cH$ containing one fixed element.

Friedgut--Kahn--Kalai--Keller reformulated this extremal-set-theoretic problem as correlation inequalities for increasing families, thereby opening a route through discrete Fourier analysis. Their program led to several conjectural strengthenings of Harris--Kleitman. The diagonal conjecture they proposed is the following sharp spectral inequality~\cite[Conjecture 5.8]{FKKK2018correlation}:
\begin{equation}\label{eq:main}
\Cov(f,g)\ge4\sum_{\varnothing\neq S\subseteq[n]}|S|\hat{f}(S)^2\hat{g}(S)^2
\end{equation}
for all increasing Boolean functions $f,g:\cube^n\to\cube$, where $\hat{f}(S)$ denotes the Fourier--Walsh coefficient of $f$. Chang and Chen~\cite{ChangChen2025submodular} verified the Friedgut--Kahn--Kalai--Keller spectral conjecture under submodularity or supermodularity assumptions, while a previous note~\cite{Chang2026spectral} proved the factor-$2$ version of~\eqref{eq:main} using reverse Bonami--Beckner hypercontractivity and Young's convolution inequality. 

Our main theorem proves this conjecture in its full generality.
\begin{theorem}\label{thm:main}
For all increasing Boolean functions $f,g:\cube^n\to\cube$,
$$
\Cov(f,g)\ge4\sum_{\varnothing\neq S\subseteq[n]}|S|\hat{f}(S)^2\hat{g}(S)^2.
$$
\end{theorem}
The factor $4$ is best possible, as equality already holds for a common
dictator $f=g=x_i$. The theorem is also sharp in a stronger structural
sense: we determine all equality cases in Theorem~\ref{thm:equality-cases}
below.

There are two useful ways to view Theorem~\ref{thm:main}. First, it is a
log-free diagonal spectral analogue of Talagrand-type correlation inequalities. Talagrand's theorem controls covariance through the coordinatewise influence overlap of $f$ and $g$, equivalently through an off-diagonal interaction between their Fourier spectra. By contrast, \eqref{eq:main} controls the diagonal collision of the full Fourier spectra, with each collision weighted by its Fourier level. Thus the estimate is most informative when the two functions have substantial Fourier mass on the same nonempty sets, especially on higher levels.

Second, the right-hand side has a natural spectral-sample interpretation. For a nonconstant Boolean function $f$, define its nonconstant spectral sample $\mathcal S_f$ by $\Prb(\mathcal S_f=S)=\frac{\hat f(S)^2}{\Var(f)},$ $\varnothing\neq S\subseteq[n].$ If $\mathcal S_f$ and $\mathcal S_g$ are independent spectral samples, then
\[
4\sum_{\varnothing\neq S\subseteq[n]}|S|\hat f(S)^2\hat g(S)^2=4\Var(f)\Var(g)\E\!\left[|\mathcal S_f|\mathbbm{1}_{\{\mathcal S_f=\mathcal S_g\}}\right].
\]
Theorem~\ref{thm:main} therefore says that positive correlation controls the degree-weighted collision probability of the two nonconstant spectral samples. This places the inequality in the same Fourier-spectral framework that underlies noise sensitivity~\cite{BKS1999,GPS2010}, percolation-type models~\cite{GS2015book,SS2010}, and recent sparse-reconstruction problems~\cite{GP2024Sparse1,GP2024Sparse2}.

Theorem~\ref{thm:main} immediately implies the unweighted diagonal conjecture of Friedgut--Kahn--Kalai--Keller~\cite[Conjecture~5.7]{FKKK2018correlation}, since $|S|\ge1$ for every nonempty $S$. We recall that, under the standard identification of $2^{[n]}$ with $\cube^n$, a family $\F\subseteq\cube^n$ is called maximal intersecting if it is not properly contained in any larger intersecting family. A maximal intersecting family is necessarily increasing, and hence its indicator function is covered by Theorem~\ref{thm:main}.
\begin{cor}[Friedgut--Kahn--Kalai--Keller~\cite{FKKK2018correlation}, Conjecture~5.7]
\label{cor:FKKK57}
Let $\F\subseteq\cube^n$ be increasing, and let $\G\subseteq\cube^n$ be maximal intersecting. Then\footnote{Here $\hat\F(S)$ and $\hat \G(S)$ denote the Fourier coefficients of $\mathbbm{1}_\F$ and $\mathbbm{1}_\G$, respectively, and $\Cov(\F,\G)$ means $\Cov(\mathbbm{1}_\F,\mathbbm{1}_\G)$.}
\begin{equation}\label{eq:FKKK57}
\Cov(\F,\G)\ge4\sum_{\varnothing\neq S\subseteq[n]}\hat \F(S)^2\hat \G(S)^2.
\end{equation}
\end{cor}

The unweighted quantity in~\eqref{eq:FKKK57} is also the nonconstant part of a convolution energy:
\[
\sum_{\varnothing\neq S\subseteq[n]}\hat f(S)^2\hat g(S)^2=\|f*g\|_2^2-\E[f]^2\E[g]^2.
\]
Thus Theorem~\ref{thm:main} gives a sharp monotone-correlation control on
convolution energy. This connects the result with discrete convolution
inequalities, hypercontractivity, and additive-energy type questions on product groups; compare the Bonami--Beckner and reverse Bonami--Beckner inequalities~\cite{Beckner1975,Bonami1970,MORSS2006} and recent sharp Young-type inequalities on the cube~\cite{beltran2025optimal}.

\begin{rmk}[Relation to Chv\'atal's conjecture]
Theorem~\ref{thm:main} settles the diagonal spectral part of the
Friedgut--Kahn--Kalai--Keller program, including Conjectures~5.7 and~5.8 of~\cite{FKKK2018correlation}. It should be distinguished from the off-diagonal inequalities in~\cite[Conjecture~5.1]{FKKK2018correlation}, which are designed to imply Chv\'atal's conjecture itself. Indeed, Conjecture~5.1(b) of Friedgut--Kahn--Kalai--Keller can be rewritten as
the following off-diagonal strengthening of Corollary~\ref{cor:FKKK57}:
\begin{equation}\label{eq:off-diagonal-FKKK}
\Cov(\F,\G)\ge 4\sum_{\substack{S,T\subseteq[n]\\S,T\neq\varnothing}}\frac{|S\cap T|}{|T|}\hat\F(S)^2\hat\G(T)^2,
\end{equation}
for increasing $\F$ and maximal intersecting $\G$. Since a maximal intersecting family has measure $1/2$, $\sum_{T\neq\varnothing}\hat\G(T)^2=\Var(\mathbbm{1}_{\G})=\frac14.$
Consequently, \eqref{eq:off-diagonal-FKKK} would imply
\[
\Cov(\F,\G)\ge \frac14\min_{i\in[n]}\Inf_i[\F],
\]
which is precisely the Friedgut--Kahn--Kalai--Keller correlation formulation of Chv\'atal's conjecture~\cite[Conjecture 1.2]{FKKK2018correlation}. Thus our result proves the sharp diagonal inequality conjectured in that program, while the corresponding off-diagonal problem remains a major open direction.
\end{rmk}

We also determine the equality cases. For an increasing Boolean function $f:\cube^n\to\cube$, let
\[
\mathcal{R}(f):=\{i\in[n]:\Inf_i[f]>0\}
\]
be the set of relevant coordinates of $f$.

\begin{theorem}[Equality cases]\label{thm:equality-cases}
Let $f,g:\cube^n\to\cube$ be increasing Boolean functions. Then equality holds in~\eqref{eq:main} if and only if one of the following alternatives holds:
\begin{enumerate}
    \item $\mathcal R(f)\cap\mathcal R(g)=\varnothing$;
    \item there is a coordinate $i\in[n]$ such that $f(x)=g(x)=x_i$;
    \item there are two distinct coordinates $i,j\in[n]$ such that, after possibly interchanging $f$ and $g$,
    \[
    f(x)=x_i x_j,
    \qquad
    g(x)=x_i\vee x_j.
    \]
\end{enumerate}
Here $x_i\vee x_j=1-(1-x_i)(1-x_j)$.
\end{theorem}

\subsection{Proof idea}

The proof is by induction on the dimension, but the induction step is not the standard two-restriction induction used for Harris--Kleitman.  Split the last coordinate and write
\[
f_0(x):=f(x,0),\quad f_1(x):=f(x,1),\qquad g_0(x):=g(x,0),\quad g_1(x):=g(x,1).
\]
The usual identity
$\Cov(f,g)=
\frac12\Cov(f_1,g_1)+\frac12\Cov(f_0,g_0)+\frac14\big(\E[f_1]-\E[f_0]\big)\big(\E[g_1]-\E[g_0]\big)
$
is well suited to proving nonnegative correlation: the first two terms are
lower-dimensional covariances, and the last term is nonnegative by monotonicity. For the spectral inequality~\eqref{eq:main}, however, this direct induction does not close. The reason is that the Fourier coefficients of $f$ are averages and differences of the coefficients of $f_0$ and $f_1$:
$$\hat f(R)=\frac{\hat f_0(R)+\hat f_1(R)}2,\qquad\hat f(R\cup\{n\})=\frac{\hat f_0(R)-\hat f_1(R)}2.$$
Thus the degree-weighted spectral term mixes the two sections. An induction using only the parallel pairs $(f_0,g_0)$ and $(f_1,g_1)$ loses the cross information contained in $(f_0,g_1)$ and $(f_1,g_0)$, producing a coefficientwise defect with no fixed sign.

The main new ingredient is a correlated four-restriction averaging.  Instead of averaging only the two parallel pairs, we apply the induction hypothesis to all four Boolean restricted pairs
$(f_0,g_0)$, $(f_0,g_1)$, 
$(f_1,g_0)$, $(f_1,g_1)$
with weights $\frac38$, $\frac18$, $\frac18$, $\frac38$. Equivalently, the two section choices are coupled with correlation \(1/2\).
This half-correlated averaging, inspired by the half-noise comparison argument
in~\cite{Chang2026spectral}, is tuned to the Fourier algebra: after averaging,
the missing mixed contribution becomes the nonnegative square
\[
    \sum_{\varnothing\neq R\subseteq[n-1]}
    |R|\bigl(\hat f_1(R)\hat g_1(R)-\hat f_0(R)\hat g_0(R)\bigr)^2 .
\]

The remaining endpoint defect involves $u=f_1-f_0$ and $v=g_1-g_0$. Since $f$ and $g$ are increasing and Boolean-valued, $0\le u,v\le1$, and the defect is controlled by the sharp convolution estimate
$$\|u*v\|_2^2\le\frac12\left(\tup{u,v}+\E[u]\E[v]\right),\qquad 0\le u,v\le1.$$
This closes the induction. The equality cases follow by tracking equality in the lower-dimensional inequalities, in the Fourier square term, and in the endpoint convolution estimate.

\medskip
\noindent\emph{Organization.}
Section~\ref{sec:prelim} collects the Fourier--Walsh notation and proves the
endpoint convolution estimate.  Section~\ref{sec:proof} proves
Theorem~\ref{thm:main} by the correlated four-restriction induction described
above, and then derives the equality classification in
Theorem~\ref{thm:equality-cases}.  Section~\ref{sec:remark} discusses further
directions, including off-diagonal extensions related to Chv\'atal's
conjecture, stability, and biased product-measure analogues.

\section{Preliminaries}\label{sec:prelim}
For $S\subseteq[n]$, define the Fourier--Walsh character by
$\chi_S(x):=(-1)^{\sum_{i\in S}x_i}.$
The family $\{\chi_S\}_{S\subseteq[n]}$ is an orthonormal basis of $L^2(\cube^n)$. The Fourier--Walsh expansion of $f:\cube^n\to\mathbb R$ is
\[
f(x)=\sum_{S\subseteq[n]}\hat f(S)\chi_S(x),
\qquad
\hat f(S)=\tup{f,\chi_S}.
\]
For more background on the Fourier–Walsh expansion the reader is referred to~\cite{ryanbook2014}. 

For $f,g:\cube^n\to\mathbb R$, define the normalized convolution on the group $\mathbb F_2^n$ by
\[
(f*g)(z):=\E_x[f(x)g(x\oplus z)],
\qquad z\in\cube^n.
\]
We shall use the following standard identity for convolution on $\mathbb F_2^n$.

\begin{fact}[Fourier transform of convolution; see~{\cite[Definition~1.24 and Theorem~1.27]{ryanbook2014}}]\label{fact:convolution-fourier}
For any $f,g:\cube^n\to\mathbb R$ and any $S\subseteq[n]$,
$\widehat{f*g}(S)=\hat f(S)\hat g(S).$
Consequently,
$\|f*g\|_2^2=
\sum_{S\subseteq[n]}\hat f(S)^2\hat g(S)^2.$
\end{fact}

\begin{lemma}\label{lem:endpoint-convolution}
Let $f,g:\cube^n\to[0,1]$. Then
\begin{equation}\label{eq:endpoint-convolution}
\|f*g\|_2^2\le \frac12\left(\tup{f,g}+\E[f]\E[g]\right).
\end{equation}
Moreover, if $f,g:\cube^n\to\cube$ are Boolean-valued, then equality holds in~\eqref{eq:endpoint-convolution} if and only if one of the following alternatives holds:
\begin{enumerate}
    \item $f\equiv0$ or $g\equiv0$;
    \item $f\equiv g\equiv1$;
    \item the sets $\supp(f)$ and $\supp(g)$ are complementary cosets of an index-two subgroup of the group $(\cube^n,\oplus)$.
\end{enumerate}
\end{lemma}
Equivalently, the third alternative means that there exist $\varnothing\neq J\subseteq[n]$ and $\alpha\in\{0,1\}$ such that
$\supp(f)=\left\{x\in\cube^n:\sum_{i\in J}x_i\equiv \alpha \pmod 2\right\}$
and
$\supp(g)=\left\{x\in\cube^n:\sum_{i\in J}x_i\equiv 1-\alpha \pmod 2\right\}.$
\begin{proof}
For $z\in\cube^n$, write $h(z):=(f*g)(z)=\E_x[f(x)g(x\oplus z)].$ Then $h(0)=\E_x[f(x)g(x)]=\tup{f,g}$. We first prove a pointwise upper bound for $h(z)$. Since $0\le f,g\le 1$, we have
$0\le h(z)\le \min\{\E[f],\E[g]\}$
for every $z\in\cube^n$. On the other hand, the inequality $f(x)g(x)\ge f(x)+g(x)-1$
holds pointwise, since it is equivalent to $(1-f(x))(1-g(x))\ge 0$. Averaging over $x$ gives $h(0)=\tup{f,g}\ge \E[f]+\E[g]-1.$
Consequently, for every $z\in\cube^n$,
$h(0)\ge \E[f]+\E[g]-1\ge 2h(z)-1.$
Hence $h(z)\le \frac{h(0)+1}{2}$. 
Combining this with $0\le h(z)\le 1$ and $h(0)\ge 0$, we obtain
\begin{equation}\label{eq:pointwise-endpoint-conv}
h(z)^2\le \frac{h(0)+1}{2}h(z)\le \frac{h(0)+h(z)}{2}.    
\end{equation}
Averaging this inequality over $z$ yields
$$\|f*g\|_2^2
=\E_z[h(z)^2]\le \frac12\left(\tup{f,g}+\E_z[h(z)]\right).$$
Moreover, $\E_z[h(z)]=\E_z\E_x[f(x)g(x\oplus z)]=\E_x [f(x)]\E_z [g(x\oplus z)]=\E[f]\E[g]$, implying that the desired inequality
$\|f*g\|_2^2\le \frac12\left(\tup{f,g}+\E[f]\E[g]\right).$

We now characterize equality in the Boolean case. Write $F:=\{x:f(x)=1\}$ and $G:=\{x:g(x)=1\}$. Then
\[
h(z)=\mu\bigl(F\cap(G\oplus z)\bigr),
\qquad
h(0)=\mu(F\cap G).
\]
Since~\eqref{eq:pointwise-endpoint-conv} holds pointwise, equality in~\eqref{eq:endpoint-convolution} holds if and only if equality holds in~\eqref{eq:pointwise-endpoint-conv} for every $z$.

Taking $z=0$, equality gives $h(0)^2=h(0)$. Hence $h(0)\in\{0,1\}$. If $h(0)=1$, then $\mu(F\cap G)=1$, so $F=G=\cube^n$. This gives $f\equiv g\equiv1$.

Now assume $h(0)=0$. Then $F\cap G=\varnothing$. Equality in
\eqref{eq:pointwise-endpoint-conv} becomes $h(z)^2=\frac12h(z)$,
so $h(z)\in\left\{0,\frac12\right\}$ for every $z\in\cube^n$. If $F=\varnothing$ or $G=\varnothing$, then $f\equiv0$ or $g\equiv0$. Thus assume that both $F$ and $G$ are nonempty. Since
$\E_z[h(z)]=\mu(F)\mu(G)>0,$
there is some $z$ with $h(z)=\frac{1}{2}$. Hence $\mu(F),\mu(G)\ge \frac{1}{2}$. But $F\cap G=\varnothing$, so $\mu(F)+\mu(G)\le1$. Consequently, $\mu(F)=\mu(G)=\frac12$ and $G=F^c$.

For every $z$, the translate $G\oplus z$ has measure $\frac{1}{2}$. If $h(z)=0$, then $G\oplus z$ is disjoint from $F$, hence $G\oplus z\subseteq G$, and therefore $G\oplus z=G$.

If $h(z)=\frac{1}{2}$, then $G\oplus z\subseteq F$, and therefore $G\oplus z=F$. Thus every translate of $G$ is either $G$ or $F$.

Let
$H:=\{z\in\cube^n:G\oplus z=G\}$
be the stabilizer of $G$ under translations, hence a subgroup of $(\cube^n,\oplus)$. Since the orbit of $G$ consists exactly of the two sets $G$ and $F$, the subgroup $H$ has index two. Moreover, $G$ is $H$-invariant and has measure $\frac{1}{2}$;
therefore $G$ is one coset of $H$, and $F=G^c$ is the other coset. This proves the necessity of the three alternatives.

Conversely, if $f\equiv0$ or $g\equiv0$, then both sides of~\eqref{eq:endpoint-convolution} are zero. If $f\equiv g\equiv1$, then both sides are one. Finally, suppose that $F=\{f=1\}$ and $G=\{g=1\}$ are complementary cosets of an
index-two subgroup. Then $\mu(F)=\mu(G)=\frac{1}{2}$, $F\cap G=\varnothing$, and for every $z$, the translate $G\oplus z$ is either $G$ or $F$. Hence
$h(z)\in\left\{0,\frac12\right\},$
with $h(z)=\frac{1}{2}$ on exactly half of the cube. Therefore $\|f*g\|_2^2=\E_z[h(z)^2]=\frac18,$ while $\frac12\left(\tup{f,g}+\E[f]\E[g]\right)=\frac{1}{8}$. Thus equality holds. This completes the proof.
\end{proof}
\begin{rmk}\label{rem:endpoint-as-noise}
Equivalently, Lemma~\ref{lem:endpoint-convolution} says
\[
\|h*k\|_2^2
\le
\left\langle h,\frac{T_0+T_1}{2}k\right\rangle,
\]
where $T_0k=\E[k]$ and $T_1k=k$. Thus the convolution energy is controlled by the average of the two endpoint noise correlations. This should be compared with the half-noise inequality from~\cite{Chang2026spectral},
\[
\tup{h,T_{1/2}k}\ge \|h*k\|_2^2,
\qquad 0\le h,k\le1.
\]
The two estimates are different; the endpoint estimate above is special to the average of $T_0$ and $T_1$ and follows from the elementary inequality $ab\ge a+b-1$ on $[0,1]$.
\end{rmk}

\section{Proof of the spectral correlation inequality}\label{sec:proof}
\subsection{Preliminaries}\label{sec:Pre}
In this subsection, we set up the notation and prove the basic lemmas needed for the inductive proof of~\cref{thm:main}. We use induction by restrictions, a standard technique in combinatorics and discrete analysis, which proceeds by fixing one coordinate and studying the corresponding restricted functions.

We first split both functions according to the last coordinate. Write a point of $\cube^n$ as $(x,z)$, where $x\in\cube^{n-1}$ and $z\in\cube$. Define the four restricted functions
$$f_0(x):=f(x,0),\quad f_1(x):=f(x,1),\quad g_0(x):=g(x,0),\quad g_1(x):=g(x,1). $$
Since $f$ and $g$ are increasing, these restricted functions are increasing Boolean functions on $\cube^{n-1}$ and $0\le f_0\le f_1\le1,0\le g_0\le g_1\le1.$ Furthermore,
\begin{equation}\label{eq:f-decomp}
f(x,z)=\frac{f_0(x)+f_1(x)}2+(-1)^z\cdot\frac{f_0(x)-f_1(x)}2,    
\end{equation}
and similarly
\begin{equation}\label{eq:g-decomp}
g(x,z)=\frac{g_0(x)+g_1(x)}2+(-1)^z\cdot\frac{g_0(x)-g_1(x)}2.   
\end{equation}

The following elementary lemma records how covariance decomposes under the last-coordinate restrictions.
\begin{lemma}[Covariance decomposition under restriction]\label{lem:covariance-restriction}
Let $n\ge1$, and let $f,g:\cube^n\to\mathbb{R}$. Define the
four restricted functions as above. Then
\begin{equation}\label{eq:covariance-restriction}
\Cov(f,g)=\Cov\left(\frac{f_0+f_1}{2},\frac{g_0+g_1}{2}\right)+\frac14\tup{f_1-f_0,g_1-g_0}.
\end{equation}
Here the covariance and inner product on the right-hand side are computed on $\cube^{n-1}$ with respect to the uniform measure.
\end{lemma}

\begin{proof}
By~\eqref{eq:f-decomp} and~\eqref{eq:g-decomp},
$$\E[fg]=\tup{\frac{f_0+f_1}{2},\frac{g_0+g_1}{2}}+\tup{\frac{f_0-f_1}{2},\frac{g_0-g_1}{2}},$$ 
since $\E_z[(-1)^z]=0,\E_z[(-1)^{2z}]=1$, and the mixed terms vanish after averaging over the last coordinate. 
Moreover,
$$\E[f]=\E\left[\frac{f_0+f_1}{2}\right],\quad\E[g]=\E\left[\frac{g_0+g_1}{2}\right].$$
Subtracting $\E[f]\E[g]$ from the identity for $\E[fg]$, we obtain
\[
\Cov(f,g)=\Cov\left(\frac{f_0+f_1}{2},\frac{g_0+g_1}{2}\right)+
\tup{\frac{f_0-f_1}{2},\frac{g_0-g_1}{2}}.
\]
This proves~\eqref{eq:covariance-restriction}.
\end{proof}
We next record the relation between the Fourier coefficients of $f$ and those of its last-coordinate restrictions.
\begin{lemma}[Fourier coefficients under restriction]\label{lem:coeff restriction}
Let $R\subseteq[n-1]$. Then
\begin{equation}\label{eq:coeff-restriction}
    \hat{f}(R)=\frac{1}{2}(\hat{f}_0(R)+\hat{f}_1(R)), \quad \hat{f}(R\cup\{n\})=\frac{1}{2}(\hat{f}_0(R)-\hat{f}_1(R)).
\end{equation}
\end{lemma}
\begin{proof}
For $R\subseteq[n-1]$, the character $\chi_R$ does not depend on the last coordinate. Therefore
\begin{equation}
\begin{split}
\hat{f}(R)&=\underset{(x,z)\sim\{0,1\}^n}{\mathbb{E}}[f(x,z)\cdot \chi_R(x,z)]\\
&=\frac{1}{2}\underset{x\sim\{0,1\}^{n-1}}{\mathbb{E}}[f_1(x)\cdot \chi_R(x)]+\frac{1}{2}\underset{x\sim\{0,1\}^{n-1}}{\mathbb{E}}[f_0(x)\cdot \chi_R(x)]=\frac{1}{2}\hat{f}_1(R)+\frac{1}{2}\hat{f}_0(R).
\end{split}
\end{equation}
On the other hand, $\chi_{R\cup\{n\}}(x,z)=\chi_R(x)\cdot(-1)^z$. Hence
\begin{equation}
\begin{split}
\hat{f}(R\cup\{n\})&=\underset{(x,z)\sim\{0,1\}^n}{\mathbb{E}}[f(x,z)\cdot \chi_{R}(x)(-1)^z]\\
&=\frac{1}{2}\underset{x\sim\{0,1\}^{n-1}}{\mathbb{E}}[f_0(x)\cdot \chi_R(x)]-\frac{1}{2}\underset{x\sim\{0,1\}^{n-1}}{\mathbb{E}}[f_1(x)\cdot \chi_R(x)]=\frac{1}{2}\hat{f}_0(R)-\frac{1}{2}\hat{f}_1(R).
\end{split}
\end{equation}   
This proves~\eqref{eq:coeff-restriction}.
\end{proof}

We shall average the induction hypothesis over a correlated choice of the four restricted pairs. Let $\varepsilon,\eta\in\{\pm1\}$ be random signs such that
\[
\E[\varepsilon]=\E[\eta]=0,
\qquad
\E[\varepsilon\eta]=\frac12.
\]
For example, one may take
\[
\Prb(\varepsilon=\eta=1)=\Prb(\varepsilon=\eta=-1)=\frac38,
\quad
\Prb(\varepsilon=1,\eta=-1)=\Prb(\varepsilon=-1,\eta=1)=\frac18.
\]
Define the random restricted functions on $\cube^{n-1}$ by
\[
f_{\varepsilon}(x):=\frac{f_0(x)+f_1(x)}2+\varepsilon\cdot\frac{f_1(x)-f_0(x)}2,\quad g_{\eta}(x):=\frac{g_0(x)+g_1(x)}2+\eta\cdot\frac{g_1(x)-g_0(x)}2.
\]
Thus $f_\varepsilon$ is either $f_0$ or $f_1$, and $g_\eta$ is either $g_0$ or $g_1$. In particular, if $f$ and $g$ are increasing Boolean functions, then each pair $(f_\varepsilon,g_\eta)$ consists of increasing Boolean functions on $\cube^{n-1}$.

The following lemma records the covariance contribution of this correlated four-restriction average.
\begin{lemma}
Let $n\ge1$, and let $f,g:\cube^n\to\mathbb{R}$. Let $\varepsilon,\eta\in\{\pm1\}$ be random signs satisfying $\E[\varepsilon]=\E[\eta]=0$ and $\E[\varepsilon\eta]=\frac12.$ Define $f_\varepsilon$ and $g_\eta$ as above. Then
\begin{equation}\label{eq:correlated-cov-average}
\E_{\varepsilon,\eta}\left[\Cov(f_\varepsilon,g_\eta)\right]=\Cov\left(\frac{f_0+f_1}{2},\frac{g_0+g_1}{2}\right)+\frac18\Cov(f_1-f_0,g_1-g_0).
\end{equation}
Here all covariances on the right-hand side are computed on $\cube^{n-1}$.
\end{lemma}

\begin{proof}
By bilinearity of covariance,
\[
\begin{aligned}
\Cov(f_{\varepsilon},g_{\eta})
&=\Cov\left(\frac{f_0+f_1}{2}+\varepsilon\,\frac{f_1-f_0}{2},
\frac{g_0+g_1}{2}+\eta\,\frac{g_1-g_0}{2}\right)  \\
&=\Cov\left(\frac{f_0+f_1}{2},\frac{g_0+g_1}{2}\right)+\frac{\varepsilon}{2}\Cov\left(f_1-f_0,\frac{g_0+g_1}{2}\right)\\
&\quad +\frac{\eta}{2}\Cov\left(\frac{f_0+f_1}{2},g_1-g_0\right)+\frac{\varepsilon\eta}{4}\Cov(f_1-f_0,g_1-g_0).
\end{aligned}
\]
Taking expectation over $(\varepsilon,\eta)$, the two linear terms vanish because $\E[\varepsilon]=\E[\eta]=0$. The last term contributes
$$\frac{\E[\varepsilon\eta]}{4}\Cov(f_1-f_0,g_1-g_0)=\frac18\Cov(f_1-f_0,g_1-g_0).$$ This proves~\eqref{eq:correlated-cov-average}.
\end{proof}

We shall also need the corresponding identity for the Fourier coefficients of the random restricted functions $f_\varepsilon$ and $g_\eta$. By Lemma~\ref{lem:coeff restriction}, for every $R\subseteq[n-1]$,
$$\hat f_\varepsilon(R)=\widehat{\frac{f_0+f_1}{2}}(R)+\varepsilon\,\widehat{\frac{f_1-f_0}{2}}(R),\quad 
\hat g_\eta(R)=\widehat{\frac{g_0+g_1}{2}}(R)+\eta\,\widehat{\frac{g_1-g_0}{2}}(R).$$
The next lemma records the effect of the correlated four-restriction average on the Fourier weights.

\begin{lemma}
Let $R\subseteq[n-1]$. Then
\begin{equation}\label{eq:fourier-correlated-average}
    \begin{split}
\E_{\varepsilon,\eta}\left[\hat f_\varepsilon(R)^2\hat g_\eta(R)^2\right]&=\widehat{\frac{f_0+f_1}{2}}(R)^2
\widehat{\frac{g_0+g_1}{2}}(R)^2\\
&+\frac14\left(\hat f_1(R)\hat g_1(R)-\hat f_0(R)\hat g_0(R)\right)^2+\frac1{16}\widehat{f_1-f_0}(R)^2\widehat{g_1-g_0}(R)^2.       
    \end{split}
\end{equation}
\end{lemma}

\begin{proof}
Fix $R\subseteq[n-1]$. Expanding the two squares and using $\E[\varepsilon]=\E[\eta]=0$ and $\E[\varepsilon\eta]=\frac12$, we get
\begin{align*}
\E_{\varepsilon,\eta}&\left[\hat f_\varepsilon(R)^2\hat g_\eta(R)^2\right]=\left(\frac{\hat f_0(R)+\hat f_1(R)}{2}\right)^2
\left(\frac{\hat g_0(R)+\hat g_1(R)}{2}\right)^2+\left(\frac{\hat f_1(R)-\hat f_0(R)}{2}\right)^2
\left(\frac{\hat g_1(R)-\hat g_0(R)}{2}\right)^2\\
&+\left[\left(\frac{\hat f_0(R)+\hat f_1(R)}{2}\right)
\left(\frac{\hat g_1(R)-\hat g_0(R)}{2}\right)+\left(\frac{\hat f_1(R)-\hat f_0(R)}{2}\right)\left(\frac{\hat g_0(R)+\hat g_1(R)}{2}\right)\right]^2.
\end{align*}
The third term simplifies as $\frac14\left(\hat f_1(R)\hat g_1(R)-\hat f_0(R)\hat g_0(R)\right)^2.$ This gives~\eqref{eq:fourier-correlated-average}.
\end{proof}

\subsection{Proof of~\cref{thm:main}}
\begin{proof}[Proof of~\cref{thm:main}]
For $f,g:\cube^n\to\mathbb{R}$, define
$$\mathcal E_n(f,g):=
\Cov(f,g)-4\sum_{\varnothing\neq S\subseteq[n]} |S|\hat f(S)^2\hat g(S)^2.$$ 
We prove $\mathcal E_n(f,g)\ge0$ for all increasing Boolean functions $f,g$ by induction on $n$. The case $n=0$ is immediate.

Assume the result is known in dimension $n-1$, and let $f_0,f_1,g_0,g_1$ be the last-coordinate restrictions introduced
in~\cref{sec:Pre}. 

First, by Lemmas~\ref{lem:covariance-restriction} and~\ref{lem:coeff restriction}, we have the exact
section decomposition
\begin{align}
\mathcal E_n(f,g)&=\mathcal E_{n-1}\left(\frac{f_0+f_1}{2},\frac{g_0+g_1}{2}\right)+\frac14\tup{f_1-f_0,g_1-g_0}\nonumber\\
&\quad-\frac14\sum_{R\subseteq[n-1]}(|R|+1)(\hat{f}_1(R)-\hat{f}_0(R))^2(\hat{g}_1(R)-\hat{g}_0(R))^2. \label{eq:section-energy}
\end{align}
We now lower-bound the first term on the right-hand side. Let $\varepsilon,\eta$ and $f_\varepsilon,g_\eta$ be as in Section~\ref{sec:Pre}. For every choice of $(\varepsilon,\eta)$, the pair $(f_\varepsilon,g_\eta)$ consists of increasing Boolean functions on $\cube^{n-1}$. Hence, by the induction hypothesis,
$$\mathcal{E}_{n-1}(f_\varepsilon,g_\eta)\ge0.$$
Averaging over $(\varepsilon,\eta)$, and using~\cref{eq:correlated-cov-average,eq:fourier-correlated-average}, gives
\begin{align}
\mathcal E_{n-1}\left(\frac{f_0+f_1}{2},\frac{g_0+g_1}{2}\right)
&\ge
\sum_{\varnothing\neq R\subseteq[n-1]}|R|
\left(\hat f_1(R)\hat g_1(R)-\hat f_0(R)\hat g_0(R)\right)^2
\nonumber \\
&\quad+
\frac14\sum_{\varnothing\neq R\subseteq[n-1]}|R|
\widehat{f_1-f_0}(R)^2\widehat{g_1-g_0}(R)^2-\frac18\Cov(f_1-f_0,g_1-g_0).
\label{eq:induction-gain}
\end{align}
Substituting~\eqref{eq:induction-gain} into~\eqref{eq:section-energy}, we obtain
\begin{align}
\mathcal E_n(f,g)
&\ge\sum_{\varnothing\neq R\subseteq[n-1]}|R|\left(\hat f_1(R)\hat g_1(R)-\hat f_0(R)\hat g_0(R)\right)^2 \nonumber\\
&\quad+\frac18\left(\tup{f_1-f_0,g_1-g_0}+\E[f_1-f_0]\E[g_1-g_0]\right)-\frac14\sum_{R\subseteq[n-1]} \widehat{f_1-f_0}(R)^2\widehat{g_1-g_0}(R)^2. \label{eq:almost-done}
\end{align}
By Fact~\ref{fact:convolution-fourier},
\[
\sum_{R\subseteq[n-1]}
\widehat{f_1-f_0}(R)^2\widehat{g_1-g_0}(R)^2
=
\|(f_1-f_0)*(g_1-g_0)\|_2^2.
\]
Since $f_1-f_0$ and $g_1-g_0$ are $\cube$-valued, Lemma~\ref{lem:endpoint-convolution} gives
\[
\frac14\|(f_1-f_0)*(g_1-g_0)\|_2^2
\le
\frac18\left(\tup{f_1-f_0,g_1-g_0}
+
\E[f_1-f_0]\E[g_1-g_0]\right).
\]
Thus the last two lines of~\eqref{eq:almost-done} are nonnegative, and the first line is a
sum of squares. Hence $\mathcal{E}_n(f,g)\ge0$, completing the induction.
\end{proof}

\subsection{Equality cases}
We now record the equality cases in Theorem~\ref{thm:main}.

We first isolate the exact nonnegative decomposition behind the proof. With the notation of
Section~\ref{sec:Pre}, set
\[
D(f_1-f_0,g_1-g_0):=\frac18\left(
\tup{f_1-f_0,g_1-g_0}+\E[f_1-f_0]\E[g_1-g_0]\right)-\frac14\|(f_1-f_0)*(g_1-g_0)\|_2^2.
\]
Then the proof of Theorem~\ref{thm:main} gives the exact identity
\begin{equation}\label{eq:exact-equality-decomposition}
\mathcal E_n(f,g)=\E_{\varepsilon,\eta}\bigl[\mathcal E_{n-1}(f_\varepsilon,g_\eta)\bigr]+D(f_1-f_0,g_1-g_0)+\sum_{\varnothing\neq R\subseteq[n-1]}|R|\left(\hat f_1(R)\hat g_1(R)-\hat f_0(R)\hat g_0(R)\right)^2.   
\end{equation}
The three terms on the right-hand side are nonnegative. The first one is nonnegative by the induction hypothesis, the second is nonnegative by Lemma~\ref{lem:endpoint-convolution}, and the third is a sum of squares.

\begin{proof}[Proof of~\cref{thm:equality-cases}]
We first check that the listed examples give equality. If
$\mathcal R(f)\cap\mathcal R(g)=\varnothing$, then $f$ and $g$ depend on disjoint sets of coordinates. Hence they are independent, so $\Cov(f,g)=0$. Moreover their nonconstant
Fourier supports are disjoint, and therefore $\mathcal{E}_n(f,g)=0$.

If $f=g=x_i$, then $\Cov(f,g)=\Var(x_i)=\frac14$, and the only nonzero nonconstant Fourier coefficient is $\hat f(\{i\})=\hat g(\{i\})=-\frac{1}{2}$. Hence $4\sum_{\varnothing\neq S}|S|\hat f(S)^2\hat g(S)^2=\frac{1}{4}$. Thus equality holds.

Finally, suppose $f=x_i x_j$ and $g=x_i\vee x_j$ for $i\neq j$. Then $fg=f$, $\E[f]=\frac{1}{4}$, and $\E[g]=\frac{3}{4}$. Hence
$\Cov(f,g)=\frac14-\frac14\cdot\frac34=\frac1{16}$. The nonzero nonconstant Fourier coefficients of $f$ on $\{i,j\}$ are
\[
\hat f(\{i\})=\hat f(\{j\})=-\frac14,
\qquad
\hat f(\{i,j\})=\frac14,
\]
while those of $g$ are
\[
\hat g(\{i\})=\hat g(\{j\})=-\frac14,
\qquad
\hat g(\{i,j\})=-\frac14.
\]
Therefore
\[
4\sum_{\varnothing\neq S}|S|\hat f(S)^2\hat g(S)^2=4\left(\frac1{16^2}+\frac1{16^2}+2\frac1{16^2}\right)=\frac1{16}.
\]
Thus equality holds in this case as well.

We now prove necessity. Assume $\mathcal E_n(f,g)=0$.

First suppose $\Cov(f,g)=0$. Applying the standard Harris--Kleitman covariance decomposition in any coordinate $i$, we get
$$
\Cov(f,g)=
\frac12\Cov(f_1,g_1)+\frac12\Cov(f_0,g_0)+\frac14\big(\E[f_1]-\E[f_0]\big)\big(\E[g_1]-\E[g_0]\big).
$$
Here the restrictions are taken in the $i$-th coordinate. All three terms on the right-hand side are nonnegative: the first two by Harris--Kleitman and the last one by monotonicity.
Hence $\bigl(\E[f_1]-\E[f_0]\bigr)\bigl(\E[g_1]-\E[g_0]\bigr)=0$. For increasing Boolean functions, $\E[f_1]-\E[f_0]=\Inf_i[f]$ and $\E[g_1]-\E[g_0]=\Inf_i[g]$. Thus $\Inf_i[f]\Inf_i[g]=0$ for every $i$, which means $\mathcal R(f)\cap\mathcal R(g)=\varnothing$.

It remains to consider the case $\Cov(f,g)>0$. Then $\mathcal R(f)\cap\mathcal R(g)\neq\varnothing$. After relabeling the coordinates, assume that $n\in\mathcal R(f)\cap\mathcal R(g)$. Thus $f_1-f_0\not\equiv0,$ and $g_1-g_0\not\equiv0.$ Using the exact decomposition~\eqref{eq:exact-equality-decomposition}, and the nonnegativity of all three terms on its right-hand side, equality $\mathcal E_n(f,g)=0$ forces
\[
D(f_1-f_0,g_1-g_0)=0.
\]
By Lemma~\ref{lem:endpoint-convolution}, since both $f_1-f_0$ and $g_1-g_0$ are nonzero $\cube$-valued functions, either $f_1-f_0\equiv g_1-g_0\equiv1$, or the supports of $f_1-f_0$ and $g_1-g_0$ are complementary cosets of an index-two subgroup of $\cube^{n-1}$.

In the first case, Booleanity gives $f_0\equiv g_0\equiv0$ and $f_1\equiv g_1\equiv1$. Hence $f(x)=g(x)=x_n$, which is the dictator case.

We now consider the second case. Let $U:=\{x\in\cube^{n-1}:f_1(x)-f_0(x)=1\}$. Since $f_0$ and $f_1$ are increasing and $f_0\le f_1$, $U$ is the difference of two increasing sets. Hence $U$ is order-convex: if $x\le y\le z$ and $x,z\in U$, then $y\in U$. Similarly, the complement $U^c=\{x\in\cube^{n-1}:g_1(x)-g_0(x)=1\}$ is also order-convex.

Since $U$ is a coset of an index-two subgroup of $\cube^{n-1}$, there is a nonempty $L\subseteq[n-1]$ and $\alpha\in\{0,1\}$ such that
$$U=\left\{x:\sum_{i\in L}x_i\equiv \alpha \pmod 2\right\}.$$
If $|L|\ge2$, choose two distinct $i,j\in L$. If $\alpha=0$, then $0\in U$ and $e_i+e_j\in U$, but $e_i\notin U$, contradicting the order-convexity of $U$. If $\alpha=1$, the same argument applies to $U^c$. Therefore $|L|=1$. Thus $U$ is a coordinate slice:
\[
U=\{x:x_k=1\}\quad\text{or}\quad U=\{x:x_k=0\}
\]
for some $k\in[n-1]$.

Suppose first that $U=\{x:x_k=1\}$. Then $f_1-f_0=1$ on $\{x_k=1\}$, so $f_0=0$ and $f_1=1$ there. If $x_k=0$, let $x^{(k)}$ be obtained from $x$ by changing the $k$-th coordinate to $1$. Since $x\le x^{(k)}$ and $f_0$ is increasing, $f_0(x)\le f_0(x^{(k)})=0$. Thus $f_0(x)=0$, and since $f_1-f_0=0$ on $\{x_k=0\}$, also $f_1(x)=0$. Consequently $f_0\equiv0$, $f_1(x)=x_k$, and therefore $f(x)=x_k x_n$. Since $g_1-g_0=\mathbbm{1}_{U^c}=\mathbbm{1}_{\{x_k=0\}}$, the same monotonicity argument gives $g_0(x)=x_k$, $g_1\equiv1$, and hence $g(x)=x_k\vee x_n$. This is the AND--OR equality case.

The case $U=\{x:x_k=0\}$ gives the same conclusion with $f$ and $g$ interchanged: $f(x)=x_k\vee x_n$, and $g(x)=x_kx_n$. This completes the proof.
\end{proof}

\section{Concluding remarks}\label{sec:remark}

\subsection{Off-diagonal inequalities and Chv\'atal's conjecture}

The most important remaining direction is to understand whether the diagonal estimate proved here can be upgraded to an off-diagonal inequality of the type proposed in~\cite{FKKK2018correlation}. Recall that Friedgut--Kahn--Kalai--Keller conjectured that for every increasing family $\F\subseteq\cube^n$ and every maximal intersecting family
$\G\subseteq\cube^n$,
\begin{equation}\label{eq:concluding-off-diagonal}
\Cov(\F,\G)\ge4\sum_{\substack{S,T\subseteq[n]\\S,T\neq\varnothing}}\frac{|S\cap T|}{|T|}\hat{\F}(S)^2\hat{\G}(T)^2 .
\end{equation}
Theorem~\ref{thm:main} proves a sharp and in fact stronger-than-needed control of the diagonal part of this inequality. The remaining difficulty is therefore genuinely off-diagonal: one must control the interaction between different Fourier supports $S$ and $T$ with $S\cap T\neq\varnothing$. A natural problem is to identify kernels $K(S,T)\ge0$ for which inequalities of the form
\[
\Cov(f,g)\ge 4\sum_{\substack{S,T\subseteq[n]\\S,T\neq\varnothing}}K(S,T)\hat f(S)^2\hat g(T)^2
\]
hold for increasing Boolean functions, or at least for the case where $g$ is maximal intersecting. The present theorem corresponds to the sharp diagonal kernel
$K(S,T)=|S|\mathbbm{1}_{S=T}.$
Finding any substantial positive off-diagonal contribution would be a meaningful step
towards the Friedgut--Kahn--Kalai--Keller program and, ultimately, towards
Chv\'atal's conjecture.

\subsection{Biased product measure}
Our method extends to the biased product setting. Let $\mu_{\mathbf p}:=\otimes_{i=1}^n\mathrm{Bern}(p_i)$, $q_i:=1-p_i$, and $s_i:=p_iq_i$, where $0<p_i<1$. We use the normalized $p$-biased Fourier basis
\[
\chi_i^{(\mathbf p)}(x):=\frac{x_i-p_i}{\sqrt{s_i}},
\qquad
\chi_S^{(\mathbf p)}(x):=\prod_{i\in S}\chi_i^{(\mathbf p)}(x),
\]
and write $\hat f_{\mathbf p}(S):=\E_{\mu_{\mathbf p}}\bigl[f\chi_S^{(\mathbf p)}\bigr]$. For further background on biased Fourier analysis on the Boolean cube, we refer the reader to~\cite[Section~8.4]{ryanbook2014}.
\begin{theorem}
For all increasing functions $f,g:(\cube^n,\mu_{\mathbf p})\to[0,1]$,
\begin{equation}\label{eq:biased-spectral}
\Cov_{\mu_{\mathbf p}}(f,g)\ge\sum_{\varnothing\neq S\subseteq[n]}\left(\sum_{i\in S}\frac1{s_i}\right)\hat f_{\mathbf p}(S)^2\hat g_{\mathbf p}(S)^2.
\end{equation}
\end{theorem}
The coefficient is sharp, since $f=g=x_i$ gives equality in~\eqref{eq:biased-spectral}.

Since the proof follows the same induction as in the uniform case, we only give the main modifications. Split the last coordinate and write $p:=p_n,q:=1-p,s:=pq$. Let $\mathbf p'=(p_1,\ldots,p_{n-1})$. For the last-coordinate restrictions
$f_0,f_1,g_0,g_1$, the biased covariance decomposition becomes
$\Cov_{\mu_{\mathbf p}}(f,g)=\Cov_{\mu_{\mathbf p'}}\bigl(qf_0+pf_1,qg_0+pg_1\bigr)+s\,\tup{f_1-f_0,g_1-g_0}_{\mu_{\mathbf p'}}.   $
Moreover, for every $R\subseteq[n-1]$,
\[
\hat f_{\mathbf p}(R)=\widehat{(qf_0+pf_1)}_{\mathbf p'}(R),
\qquad
\hat f_{\mathbf p}(R\cup\{n\})=\sqrt{s}\,\widehat{(f_1-f_0)}_{\mathbf p'}(R),
\]
and similarly for $g$.
Define the biased energy
$$\mathcal E_{\mathbf p}(f,g):=\Cov_{\mu_{\mathbf p}}(f,g)-\sum_{\varnothing\neq S\subseteq[n]}\left(\sum_{i\in S}\frac1{s_i}\right)\hat f_{\mathbf p}(S)^2\hat g_{\mathbf p}(S)^2.$$
The section decomposition gives
\begin{align}
\mathcal E_{\mathbf p}(f,g)
&=
\mathcal E_{\mathbf p'}\bigl(qf_0+pf_1,qg_0+pg_1\bigr)
+
s\,\tup{f_1-f_0,g_1-g_0}_{\mu_{\mathbf p'}}
\nonumber\\
&\quad
-
s\sum_{R\subseteq[n-1]}
\widehat{(f_1-f_0)}_{\mathbf p'}(R)^2
\widehat{(g_1-g_0)}_{\mathbf p'}(R)^2
\nonumber\\
&\quad
-
s^2
\sum_{\varnothing\neq R\subseteq[n-1]}
\left(\sum_{i\in R}\frac1{s_i}\right)
\widehat{(f_1-f_0)}_{\mathbf p'}(R)^2
\widehat{(g_1-g_0)}_{\mathbf p'}(R)^2.
\label{eq:biased-section-energy}
\end{align}
Here and below, all inner products and Fourier coefficients on the right-hand side are taken with respect to $\mu_{\mathbf p'}$.

To lower-bound the first term in~\eqref{eq:biased-section-energy}, choose a coupling $(Z,W)\in\{0,1\}^2$ with marginals $\mathrm{Bern}(p)$ and
\[
\Prb(Z=W=0)=q^2+\frac{s}{2},\ \Prb(Z=W=1)=p^2+\frac{s}{2},\ \Prb(Z=0,W=1)=\Prb(Z=1,W=0)=\frac{s}{2}.
\]
Equivalently,
$\E[Z]=\E[W]=p$ and $\E[(Z-p)(W-p)]=\frac{s}{2}.$
Apply the induction hypothesis to the four restricted pairs
$(f_0,g_0)$, $(f_0,g_1)$, $ (f_1,g_0)$, $(f_1,g_1)$
with these coupling weights. A direct expansion gives, for every $R\subseteq[n-1]$,
\begin{align}
&\E_{Z,W}\bigl[\hat f_Z(R)^2\hat g_W(R)^2\bigr]
\nonumber=
\widehat{(qf_0+pf_1)}(R)^2
\widehat{(qg_0+pg_1)}(R)^2
\nonumber\\
&\quad+
s\left[
\widehat{(f_1-f_0)}(R)\widehat{(qg_0+pg_1)}(R)
+
\widehat{(qf_0+pf_1)}(R)\widehat{(g_1-g_0)}(R)
+
\frac{q-p}{2}\widehat{(f_1-f_0)}(R)\widehat{(g_1-g_0)}(R)
\right]^2
\nonumber\\
&\quad+
\frac{s}{4}
\widehat{(f_1-f_0)}(R)^2
\widehat{(g_1-g_0)}(R)^2.
\label{eq:biased-four-restriction-average}
\end{align}
This is the biased analogue of~\eqref{eq:fourier-correlated-average} in the uniform proof.

Combining \eqref{eq:biased-section-energy} and \eqref{eq:biased-four-restriction-average}, we obtain
\[
\mathcal E_{\mathbf p}(f,g)\ge\text{a sum of squares}
+s\left(\frac14-s\right)\sum_{\varnothing\neq R\subseteq[n-1]}\left(\sum_{i\in R}\frac1{s_i}\right)\widehat{(f_1-f_0)}(R)^2\widehat{(g_1-g_0)}(R)^2
\]
\[
\qquad
+s\left[\frac12\left(\tup{f_1-f_0,g_1-g_0}+\E[f_1-f_0]\E[g_1-g_0]\right)-\sum_{R\subseteq[n-1]}\widehat{(f_1-f_0)}(R)^2\widehat{(g_1-g_0)}(R)^2\right].
\]
The second term is nonnegative because $s=pq\le\frac{1}{4}$. The last bracket is nonnegative by
the biased endpoint estimate
\[
\sum_{R\subseteq[n-1]}\hat u_{\mathbf p'}(R)^2\hat v_{\mathbf p'}(R)^2\le\frac12\left(\tup{u,v}_{\mu_{\mathbf p'}}+\E_{\mu_{\mathbf p'}}[u]\E_{\mu_{\mathbf p'}}[v]\right),
\qquad
0\le u,v\le1.
\]
This estimate is the biased analogue of Lemma~\ref{lem:endpoint-convolution}. Therefore $\mathcal E_{\mathbf p}(f,g)\ge0$, proving~\eqref{eq:biased-spectral}.

\subsection{Gaussian and invariance-principle analogues}

A further direction is to look for continuous analogues.  Let $\gamma_n$ be standard Gaussian measure on $\mathbb R^n$, and let
\[
    F=\sum_{\alpha\in\mathbb N^n}\hat F(\alpha)H_\alpha,
    \qquad
    G=\sum_{\alpha\in\mathbb N^n}\hat G(\alpha)H_\alpha
\]
be the Hermite expansions of two coordinatewise increasing functions. For further background on Hermite analysis over Gaussian space, we refer the reader to~\cite[Section~11]{ryanbook2014}. A natural Gaussian analogue would be an inequality of the form
\begin{equation}\label{eq:gaussian-open}
    \Cov_{\gamma_n}(F,G)
    \stackrel{?}{\ge}
    \sum_{\alpha\neq0}
    |\alpha|\hat F(\alpha)^2\hat G(\alpha)^2,
\end{equation}
under suitable boundedness or integrability assumptions.  The weight $|\alpha|$ is the eigenvalue of the Ornstein--Uhlenbeck generator, and therefore plays the same role as the Fourier level $|S|$ on the discrete cube.

An inequality such as \eqref{eq:gaussian-open} would be a spectral strengthening of Gaussian positive association and would complement quantitative Gaussian correlation inequalities, including the results of Royen~\cite{Royen2014} and the quantitative framework of De--Nadimpalli--Servedio~\cite{DNS2022}. It would also be interesting to understand whether an invariance principle~\cite{MMO2010} can transfer suitable low influence versions of Theorem~\ref{thm:main} to Gaussian space.

\noindent\textbf{Acknowledgements.} During an early exploratory stage, the authors used language-model-based
tools to help brainstorm candidate proof strategies; any suggestions arising from these tools served
only as informal inspiration.

\bibliographystyle{abbrv}
\bibliography{reference}

\end{document}